\documentclass[11pt]{article}

\usepackage{amsmath}
\usepackage{amsthm}

\theoremstyle{plain}
\newtheorem{thm}{Theorem}
\newtheorem{lem}{Lemma}

\newtheorem{con}{Conjecture}

\theoremstyle{definition}
\newtheorem{defn}{Definition}

\newcommand{\cell}[2]{\ensuremath{[#1,\: #2]}}

\usepackage{graphicx}

\usepackage[implicit=false,colorlinks=true,urlcolor=red]{hyperref}

\begin{document}

\title{The Pascal Rhombus and the Stealth Configuration}   
\author{Paul K. Stockmeyer\\Department of Computer Science\\The College of William and Mary\\Williamsburg, Virginia, USA%
\\ Email: \href{mailto:stockmeyer@cs.wm.edu}{stockmeyer@cs.wm.edu}}
\date{} 
\maketitle

\begin{abstract}
The Pascal rhombus is a variant of Pascal's triangle in which each term is a sum of four earlier terms.  
Klostermeyer et al.\ made
four conjectures about the Pascal rhombus modulo 2.  In this paper we show how exploration of the stealth shape leads to
unified proofs of all of these conjectures.
\end{abstract}

\section{Introduction}     
\label{intro}
The Pascal rhombus was introduced in 1997 by Klostermeyer {\it et al.}~\cite{Klost} as a variant of Pascal's triangle.  
The term rhombus does not refer to this figure's shape---as with Pascal's triangle, the Pascal rhombus forms an infinite triangular wedge.  Rather, the term refers to the rule for recursively constructing the figure. 
Informally, each element in the
Pascal rhombus is the sum of three adjacent elements in the row immediately above the element, plus one element from two rows above.
The new element, together with the four elements that contribute to it, lie in the shape of a rhombus.
The figure begins with a single 1 in the top row.  Figure \ref{rhom} shows the first six rows of the Pascal rhombus.

\begin{figure}[h!]
\centerline{\includegraphics[width=3.3in,height=2.1in,keepaspectratio]{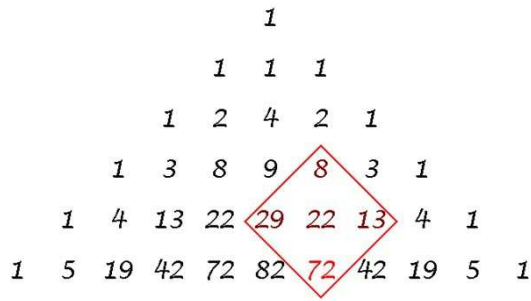}}
\caption{The Pascal rhombus.}
\label{rhom}
\end{figure}

More formally, we number the rows and columns of the Pascal rhombus, with \cell{n}{k} denoting the cell in  
row $n$, column $k$, and $R\cell{n}{k}$ denoting its contents.  
The non-zero elements are those for which $-n< k < n$.  The initial conditions are that $R\cell{0}{j}=0$ for all $j$, 
$R\cell{1}{j} = 0$ for $j\ne0$, and 
$R\cell{1}{0}=1$.  Construction
continues with
\begin{equation*}
R\cell{n}{k} = R\cell{n-1}{k-1} + R\cell{n-1}{k} + R\cell{n-1}{k+1} + R\cell{n-2}{k}
\end{equation*}
for $n\ge 2$. 
(Note that this indexing is different from 
that of \cite{Klost}.)

Near the end of \cite{Klost} the authors consider the Pascal rhombus with entries reduced modulo 2. 
Figure \ref{mod2} shows the first
32 rows of the Pascal rhombus (mod 2), with odd entries in black and even entries in white. The construction rule now becomes:
A cell at the bottom of a 5-cell rhombus is colored black if and only if an odd number of the other four cells are colored black.
Thus every 5-cell rhombus contains an even number of black cells.  The only exception to this rule in the entire infinite grid is
the rhombus centered at cell \cell{0}{0} which contains the single black cell at position \cell{1}{0} that initiates 
the construction.

\begin{figure}[h!]
\centerline{\includegraphics[width=3.89in,height=2.1in,keepaspectratio]{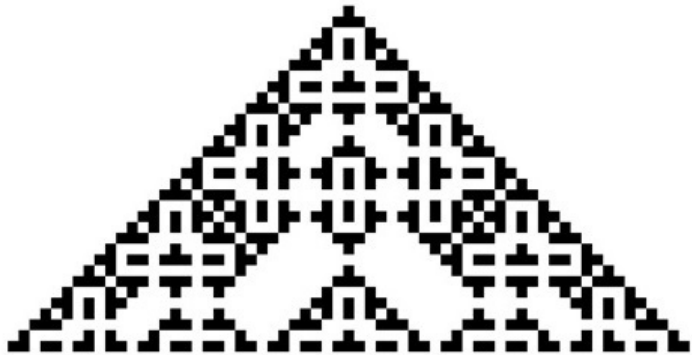}}
\caption{The Pascal rhombus (mod 2)}
\label{mod2}
\end{figure}

This is of course quite reminiscent of the well-known Pascal's triangle (mod 2), where the first $2^n$ rows form a structure
consisting of three copies of the $2^{n-1}$ row structure.  Properly scaled, these structures approach the fractal
known as the Sierpi\'{n}ski gasket.  See, for example, \cite[Section 1.2]{edgar}.
No doubt motivated by the properties of this familiar triangle, Klostermeyer {\it et al.}~\cite{Klost}
made four conjectures about the structure of the Pascal rhombus (mod 2).  

\begin{con}[Klostermeyer {\it et al.}\ \cite{Klost}]\label{con1}  For any $n\ge 1$
the triangle in the Pascal rhombus {\rm (mod 2)} with corners at \cell{1}{0}, \cell{2^{n-1}}{ -(2^{n-1}-1)}, 
and \cell{2^{n-1}}{2^{n-1}-1}
is identical to the triangle with corners at \cell{2^n+1}{-2^n}, \cell{2^n + 2^{n-1}}{-(2^n + 2^{n-1}-1)},
and \cell{2^n + 2^{n-1}}{-(2^n - 2^{n-1}+1)},  and also identical to the triangle with corners at
\cell{2^n+1}{2^n}, \cell{2^n + 2^{n-1}}{ 2^n - 2^{n-1}+1},
and \cell{2^n + 2^{n-1}}{2^n + 2^{n-1}-1}.
\end{con}

For example, with $n=4$ in Figure~\ref{mod2}, the triangle with corners at \cell{1}{0}, \cell{8}{-7} and \cell{8}{7}
(the top quarter of the figure) is identical to the triangle with corners at \cell{17}{-16}, \cell{24}{-23} and
\cell{24}{-9} (the third quarter down, on the left), and to the triangle with corners at \cell{17}{16}, \cell{24}{9},
and \cell{24}{23} (the third quarter down, on the right).  The conjecture makes no mention of a fourth such
triangle, with corners at \cell{25}{0}, \cell{32}{ -7} and \cell{32}{7} (the middle of the bottom quarter of the figure).

\begin{con}[Klostermeyer {\it et al.}\ \cite{Klost}]\label{con2}
Let $I_n$ be the number of ones in row $m=2^n$ of the Pascal rhombus {\rm (mod 2)}.  Then
\begin{equation*}
I_n = \frac13\left( 2^{n+2} - (-1)^n\right).
\end{equation*}
\end{con}

\begin{con}[Klostermeyer {\it et al.}\ \cite{Klost}]\label{con3}
For each $k\ge 1$, diagonal $D_k$ of the Pascal rhombus {\rm (mod 2)} consists of cells \cell{n}{ -n+k+1} for 
$n\ge (k+1)/2$. Each such diagonal
is periodic with 
period length $2^p$, where $p = \lceil \log_2(k)\rceil + 1$.  The period of the \cell{n}{ -n+1} diagonal $D_0$ is 1.
\end{con}

Note that the periods claimed here are not necessarily  minimal.  The authors illustrate the conjecture by observing that
diagonal $D_6$ begins
11011011, and presumably repeats these eight values.  Conjecture \ref{con3}, however, only asserts a period of 16.

\begin{con}[Klostermeyer {\it et al.}\ \cite{Klost}]\label{con4}
Let $G_n$ and $H_n$ be the number of odd and even entries, respectively, in the first $n$ rows of the Pascal rhombus.
Then $\lim_{n\rightarrow \infty}G_n/H_n = 0$.
\end{con}

In a subsequent paper by many of the same authors, Goldwasser {\it et al.} \cite{Gold} proved the correctness of
Conjectures \ref{con2} and \ref{con4}, using an elaborate decomposition of the Pascal rhombus based on odd and even
values of the indices $n$ and $k$.  In the next two sections we show how one key observation leads to unified proofs of
all four conjectures.

\section{The stealth shape}
\label{shape}

A standard way to analyze structures such as the Pascal triangle (mod 2) and their fractal limits is to exploit the
fact that such structures can be decomposed into a finite union of lower order structures.
Unfortunately, there is no way to decompose the first $2^n$ rows of the Pascal rhombus (mod 2) into a finite number of 
smaller, similar copies of these triangles. Conjecture \ref{con1} suggests how far we are from this desirable situation.   However, truncating the Pascal rhombus (mod 2) along a zig-zag line produces a
structure that {\it can} be decomposed into five lower order copies of itself.  We call this resulting structure a
stealth configuration because of its vague resemblance to the B2 stealth bomber.  See Figure \ref{stealth}.

\begin{figure}[h!]
\centerline{\includegraphics[width=2.95in,height=2.21in,keepaspectratio]{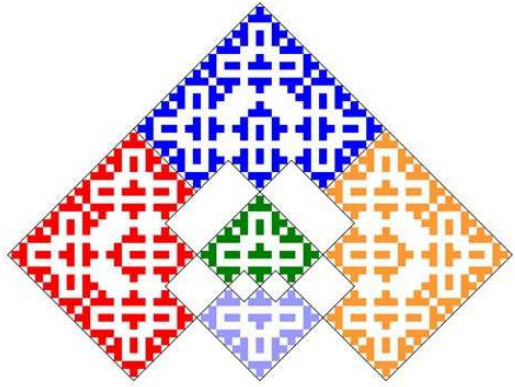}}
\caption{The Stealth Configuration and Stealth Decomposition.}
\label{stealth}
\end{figure}

\begin{defn}
\label{stealthdef}
The {\it order $n$ stealth configuration} $S_n$, for $n\ge 2$, is that part of the Pascal rhombus (mod 2) strictly within 
the octagon with edges linking
cells \cell{0}{0}, 
\cell{2^n}{-2^n}, \cell{2^n + 2^{n-1}}{-2^{n-1}}, \cell{2^n + 2^{n-2}}{-2^{n-2}}
\cell{2^n + 2^{n-1}}{0}, \cell{2^n + 2^{n-2}}{2^{n-2}}, \cell{2^n + 2^{n-1}}{2^{n-1}},
\cell{2^n}{2^n}, and back to \cell{0}{0}.  
The order 1 stealth configuration $S_1$ consists of cells \cell{1}{0}, \cell{2}{-1}, \cell{2}{0} and \cell{2}{1}
in the Pascal rhombus (mod 2), and the order 0 stealth configuration $S_0$ consist of just cell \cell{1}{0}.
\end{defn}

We will refer to translations and rotations of stealth configurations as stealth configurations as well, and 
will sometimes refer to
those growing down from cell \cell{1}{0} as being in standard position.

Figure \ref{stealth} shows the stealth configuration $S_5$, 
strictly inside the octagon passing through 
\cell{0}{0}, \cell{32}{-32}, \cell{48}{-16}, \cell{40}{-8}, \cell{48}{0}, \cell{40}{8}, 
\cell{48}{16}, \cell{32}{32}, and back to \cell{0}{0}.
It consists of an order 4 stealth configuration growing down from \cell{1}{0} (the nose), an order 4
stealth configuration growing to the right from \cell{32}{-31} (the left wing), an order 4
stealth configuration growing to the left from  \cell{32}{31} (the right wing), an order 3
stealth configuration growing down from \cell{25}{0} (the main body), and an order 3 stealth
configuration growing up from  \cell{47}{0} (the tail).  

\begin{thm}
\label{main}
For all $n\ge 2$, the order $n$ stealth configuration $S_n$ in standard position can be decomposed into the disjoint 
union of 5 smaller
stealth configurations:  an order $n-1$ stealth configuration growing down from  \cell{1}{0} (the nose),
an order $n-1$ stealth configuration growing to the right from \cell{2^n}{-(2^n-1)} (the left wing),
an order $n-1$ stealth configuration growing to the left from  \cell{2^n}{2^n-1} (the right wing),
an order $n-2$ stealth configuration growing down from  \cell{2^{n}-2^{n-2}+1}{0} (the main body), and 
an order $n-2$ stealth configuration growing up from \cell{2^n+2^{n-1}-1}{0} (the tail).
\end{thm}

We need two temporary definitions before we prove Theorem \ref{main}.  First,
the {\it order $n$ pseudo-stealth configuration} $S_n^{\prime}$ is defined as follows:  For $n=0$ and $n=1$, the order $n$
pseudo-stealth configuration $S_n^{\prime}$ is the same as the order $n$ stealth configuration $S_n$ from Definition \ref{stealthdef}.
For $n\ge 2$, $S_n^{\prime}$ is the configuration formed from copies of lower order pseudo-stealth configurations 
$S_{n-1}^{\prime}$ and $S_{n-2}^{\prime}$ in accordance with
Theorem \ref{main}. Our task is to prove that the cut-and-paste pseudo-stealth configuration $S_n^{\prime}$ is identical to
the  rhombus rule stealth configuration $S_n$ for all $n\ge 0$.

Second, we call a grid cell {\it exceptional} with respect 
to a stealth or pseudo-stealth configuration if it is the center cell of a 5-cell rhombus containing an odd number of ones.
We will show, for example, that
that the exceptional cells with respect to the order 5 configuration 
in Figure \ref{stealth} are \cell{0}{0}. \cell{32}{-32}, \cell{48}{-16}, \cell{48}{0}, \cell{48}{16},
and \cell{32}{32} (the acute corners of the bounding octagon).

\begin{lem}
\label{excep}
The exceptional cells  with respect to the order $n$ pseudo-stealth configuration $S_n^{\prime}$ in standard position, for  $n\ge 1$, are at
\cell{0}{0}, \cell{2^n}{-2^n}, \cell{2^n+2^{n-1}}{-2^{n-1}}, \cell{2^n + 2^{n-1}}{0}, \cell{2^n + 2^{n-1}}{2^{n-1}},
and \cell{2^n}{2^n}.
\end{lem}

\begin{proof}[Proof of Lemma \ref{excep}]
The proof is by induction.  The cases $n=1$ and $n=2$ can be confirmed by hand.  We assume that the statement is
true for all pseudo-stealth configurations of order less than $n$, for some arbitrary $n\ge 3$, and consider an
order $n$ pseudo-stealth configuration $S_n^{\prime}$ of order $n$, in standard position.
Configuration $S_n^{\prime}$ is the union of
the following: An order $n-1$ pseudo-stealth configuration (the nose) with exceptional cells at 
\cell{0}{0}, \cell{2^{n-1}}{-2^{n-1}}, 
\cell{2^{n} - 2^{n-2}}{-2^{n-2}}, \cell{2^{n} - 2^{n-2}}{0}, \cell{2^{n}-2^{n-2}}{2^{n-2}}, and \cell{2^{n-1}}{2^{n-1}};
an order $n-1$ pseudo-stealth configuration (the left wing) with exceptional cells at 
\cell{2^n}{-2^n}, \cell{2^n +2^{n-1}}{-2^{n-1}},
\cell{2^n + 2^{n-2}}{-2^{n-2}}, \cell{2^n}{-2^{n-2}}, \cell{2^{n} - 2^{n-2}}{-2^{n-2}}, and \cell{2^{n-1}}{-2^{n-1}}; 
an order $n-1$ pseudo-stealth configuration (the right wing) with exceptional cells at 
\cell{2^n}{2^n}, \cell{2^{n-1}}{2^{n-1}},
\cell{2^{n}-2^{n-2}}{2^{n-2}}, \cell{2^n}{2^{n-2}}, \cell{2^n+2^{n-2}}{2^{n-2}}, and \cell{2^n+2^{n-1}}{2^{n-1}};
an order $n-2$ pseudo-stealth configuration (the main body) with exceptional cells at 
\cell{2^n - 2^{n-2}}{0}, \cell{2^n}{-2^{n-2}},
\cell{2^n+2^{n-3}}{-2^{n-3}}, \cell{2^n+2^{n-3}}{0}, \cell{2^n+2^{n-3}}{2^{n-3}}, and \cell{2^n}{2^{n-2}}; and on order $n-2$
pseudo-stealth configuration (the tail) with exceptional cells at 
\cell{2^n + 2^{n-1}}{0}, \cell{2^n + 2^{n-2}}{2^{n-2}},
\cell{2^n + 2^{n-3}}{2^{n-3}}, \cell{2^n + 2^{n-3}}{0}, \cell{2^n + 2^{n-3}}{-2^{n-3}}, and \cell{2^n + 2^{n-2}}{-2^{n-2}}.

It is easy to see that a cell in the union of several pseudo-stealth configurations is exceptional if and only if it
is exceptional in an odd number of the component pseudo-stealth configurations.  Of the 30 exceptional cells listed above,
24 pair off and cancel out, leaving configuration $S_n^{\prime}$ with just the six exceptional cells  specified in the statement 
of the lemma.
\end{proof}

\begin{proof}[Proof of Theorem \ref{main}]
We will show that for all $n\ge 0$, the order $n$ stealth configuration $S_n$ is identical to the order $n$ 
pseudo-stealth configuration $S_n^{\prime}$.
For $n=0$ and $n=1$, this is true by definition.  
For arbitrary $n\ge 2$, let $S_{n+1}^{\prime}$ be the order $n+1$ pseudo-stealth configuration in standard position.
Consider the nose section of $S_{n+1}^{\prime}$. On the one hand, from Lemma \ref{excep} we know all the exceptional cells
of $S_{n+1}^{\prime}$, and all the cells in the nose of $S_{n+1}^{\prime}$ are closer to the exceptional cell \cell{0}{0} than to
any other.  Thus the nose of $S_{n+1}^{\prime}$ is contained in the Pascal Rhombus (mod 2) that grows down from 
cell \cell{1}{0} without interruption, using the rhombus rule.  That is, it is the order $n$ stealth configuration $S_n$.
On the other hand, the nose of $S_{n+1}^{\prime}$ is, by construction,  the order $n$ pseudo-stealth configuration 
$S_n^{\prime}$ as well.
\end{proof}

\section{Proving the conjectures}
\label{conj}
Armed with Theorem \ref{main}, we can readily prove the conjectures of Klostermeyer {\it et al.}\cite{Klost}
stated in Section \ref{intro}.

\begin{proof}[Proof of Conjecture \ref{con1}]
Let $S_{n+1}$ be the order $n+1$ stealth configuration in standard position for some
$n\ge 1$.  It contains
an order $n-1$ stealth configuration growing down from \cell{1}{0} (the nose of the nose of $S_{n+1}$),
an order $n-1$ stealth configuration growing down from \cell{2^{n}+1}{-2^{n}} (the right wing 
component of the left wing of $S_{n+1}$),
and an order $n-1$ stealth configuration growing down from \cell{2^{n}+1}{2^{n}} (the left wing 
component of the right wing of $S_{n+1}$).
The three triangles of Conjecture \ref{con1} consist of the first $2^{n-1}$ rows of these three order $n-1$ stealth
configurations.
\end{proof}

The fourth identical triangle noted in Section \ref{intro} consists of the first $2^{n-1}$ rows of the order $n-2$
stealth configuration growing down from \cell{2^n + 2^{n-1}+1}{0} (the body of configuration $S_{n+1}$).

\begin{proof}[Proof of Conjecture \ref{con2}]
For $n\ge 1$ let $A_n$ denote the number of ones (black cells) in the central
column of the
order $n$ stealth configuration in standard position, from \cell{1}{0} down to \cell{2^n + 2^{n-1}}{0}.  
By inspection we have $A_1 = 2$
and $A_2 = 4$.  From Theorem \ref{main} we have that $A_n = A_{n-1} + 2A_{n-2}$ for $n\ge 3$.  
It is easy to see that 
\begin{equation}
A_n = 2^n
\end{equation}
is the solution to this recurrence relation.  For completeness we set $A_0 = 1$.  
This is sequence \href{https://oeis.org/A000079}{A000079} in the OEIS \cite{OEIS}.

Now for $n\ge 1$ let $B_n$ denote the number of ones (black cells) on the row from left wingtip to right wingtip
of the order $n$ stealth configuration in standard position, from \cell{2^n}{-(2^n-1)} to  \cell{2^n}{2^n-1}.
By inspection we have $B_1 = 3$ and $B_2 = 5$.  From Theorem \ref{main} we have that $B_n = 2A_{n-1} + B_{n-2}$ for $n\ge 3$,
or $B_n = 2^n + B_{n-2}$.  Standard methods for recurrence relations yield the solution
\begin{equation}
B_n = \frac13\left( 2^{n+2} - (-1)^n\right),
\end{equation} 
which can be confirmed by induction.  For completeness we set $B_0= 1$. This result was first derived, using different methods
and different recurrences, by Goldwasser {\it et al.} \cite{Gold} as a special case of their Theorem 1.
This is sequence \href{https://oeis.org/A001045}{A001045} in the OEIS; its first few terms are 
1, 3, 5, 11, 21, 43,  85, 171, \dots\ .
\end{proof}

Let $S_n$ be the stealth configuration of order $n\ge 1$.
We call  the set of cells strictly within the rhombus linking cells \cell{0}{0}, 
\cell{2^{n-1}}{-2^{n-1}}, \cell{2^n}{0}, and \cell{2^{n-1}}{2^{n-1}} the {\it upper rhombus} of $S_n$.  Likewise, we call  the set of cells strictly within the rhombus linking cells 
\cell{2^{n-1}}{-2^{n-1}}, \cell{2^n}{-2^n}, \cell{2^n + 2^{n-1}}{-2^{n-1}}, and \cell{2^n}{0} the {\it left rhombus} of $S_n$.

\begin{lem}
\label{mirror}
For all $n\ge 1$, the upper rhombus of  stealth configuration $S_n$ is the mirror image of the left rhombus of $S_n$,
reflected along the line of cells \cell{2^{n-1}+k}{-2^{n-1}+k} for $0\le k\le 2^{n-1}$.
\end{lem}

\begin{proof}
The proof is by induction on $n$.  The claim is easily confirmed for $n= 1$ and 2.  We suppose the claim is true for all
stealth configurations of order less than $n$, for some arbitrary $n\ge 2$, and examine the stealth configuration $S_n$ of 
order $n$.  Clearly the upper rhombus of $S_n$ consists of the order $n-1$ stealth configuration forming the nose section of
$S_n$, together  with  the upper rhombus of the order $n-2$ configuration forming the main body section of $S_n$.  Likewise, the
left rhombus of $S_n$ consists of the order $n-1$ configuration forming the left wing section of $S_n$, together with the
left rhombus of the order $n-2$ configuration forming the main body section of $S_n$.  The basic bilateral symmetry of
stealth configurations insures that the nose section of $S_n$ reflects onto the left wing section of $S_n$, and the inductive
hypothesis insures that the upper rhombus of the main body section of $S_n$ reflects onto the left rhombus of the 
main body section of $S_n$.
\end{proof}

\begin{proof}[Proof of Conjecture \ref{con3}]
We will prove the somewhat stronger claim that for all $k\ge 1$, diagonal $D_k$ is periodic with period $2^{m+1}$, 
where $m=\lfloor \log_2(k)\rfloor$.  
In particular, we will show that within every stealth configuration $S_{m+n}$ with $n\ge 1$, diagonal $D_k$ is both periodic with period $2^{m+1}$ and a palindrome.  
The proof is by induction on $n$.

Let $D_k$ be a diagonal, with $m=\lfloor \log_2(k)\rfloor$.  We assume $k$ is even; the odd case is slightly different and
left to the reader.  $D_k$ contains $2^m$ cells within the upper rhombus of configuration $S_{m+1}$.  These are reflected
onto the next $2^m$ cells of $D_k$, lying within the left rhombus of $S_{m+1}$.  Together these $2^{m+1}$ cells, the
intersection of $D_k$ and $S_{m+1}$, form a
palindrome and by default are periodic of period $2^{m+1}$.  This establishes the claim within  $S_{m+1}$.

Now assume the claim is true within configuration $S_{m+n}$ for some arbitrary $n\ge 1$.  
By hypothesis, we know that within configuration
$S_{m+n}$,  diagonal $D_k$ is a palindrome and periodic of period $2^{m+1}$.  Now configuration $S_{m+n}$ is contained within
the upper rhombus of configuration
$S_{m+n+1}$ and thus reflects onto the left rhombus of $S_{m+n+1}$ forming a palindrome twice as long.  But the mirror image of a
palindrome is itself, so the intersection of $D_k$ and $S_{m+n+1}$ is just the doubling of $D_k$ within $S_{m+n}$.  It
is both periodic with period $2^{m+1}$ and a palindrome, confirming that the claim is true within $S_{m+n+1}$.
\end{proof}

Before proving Conjecture \ref{con4}, we examine the density of ones in the order $n$ stealth configuration.

\begin{thm}
\label{density}
The density of ones in the order $n$ stealth configuration approaches a limit of 0 as $n$ approaches infinity.
\end{thm}

\begin{proof}
Let $C_n$ denote the total number of ones (black entries) in the order $n$ stealth
configuration.  By inspection we have $C_0 = 1$ and  $C_1 = 4$. From Theorem \ref{main} we know that $C_n = 3C_{n-1} + 2C_{n-2}$
for all $n\ge 2$.  The solution to this recurrence relation is
\begin{equation}
C_n = \frac{17 + 5\sqrt{17}}{34}\left(\frac{3+\sqrt{17}}{2}\right)^n + \frac{17 - 5\sqrt{17}}{34}
\left(\frac{3 - \sqrt{17}}{2}\right)^n
\end{equation}
which can be confirmed by induction.  This is sequence \href{https://oeis.org/A055099}{A055099} in the OEIS;  its first few terms are 1, 4, 14, 50, 178, 634, 2258, 8042, \ldots\ .

Now let $D_n$ denote the total number of entries, both zeros and ones, in the order $n$ stealth configuration $S_n$. 
Note that a triangle of $k$
rows contains $1+3+5+\cdots + (2k-1) = k^2$ cells. For $n\ge 2$, the top $2^n$ rows of $S_n$
contain $(2^n)^2$ cells.  The bottom part of $S_n$ consists of a triangle of $2^n-1$ rows, minus a triangle of
$2^{n-1}$ rows, minus two triangles of $2^{n-2}$ rows.  Adding and subtracting all the parts, we have
\begin{equation}
\begin{split}
D_n &= (2^n)^2 + (2^n-1)^2 - (2^{n-1})^2 - 2(2^{n-2})^2\\
     &= \frac{13}{8}\left(4^{n}\right) - 2\left(2^{n}\right) + 1
\end{split}
\end{equation}
for $n\ge 2$, with $D_0 =  1$  and $D_1 = 4$.
This is sequence \href{https://oeis.org/A256959}{A256959} in the OEIS;  its first few terms are 1, 4, 19, 89, 385, 1601, 6529, 26369, \ldots\ .

The density of the order $n$ stealth configuration is $C_n/D_n$, which clearly approaches a limit of 0.
\end{proof}

\begin{proof}[Proof of Conjecture \ref{con4}]

Let $E_n$ denote the number of ones in the first $2^n$ rows of the order $n$ stealth configuration.  
From Theorem \ref{main} we know that this triangle consists of an order
$n-1$ stealth configuration for its nose section, containing $C_{n-1}$ ones; the first $2^{n-2}$ rows of an order $n-2$ stealth configuration for its
middle section, containing $E_{n-2}$ ones;  approximately half of an order $n-1$ stealth configuration for its left part; and
approximately half of an order $n-1$ stealth configuration for its right part.  We say ``approximately'' here because both
the left part and the right part contain the spine of the order $n-1$ stealth configuration, so the left and right parts together contain $C_{n-1} + A_{n-1}$ ones. See Figure \ref{triangle}.

\begin{figure}[h!]
\centerline{\includegraphics[width=2.47in,height=1.31in,keepaspectratio]{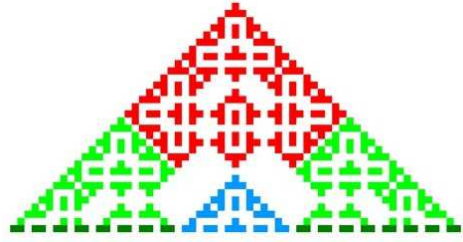}}
\caption{Decomposing the first $2^n$ rows}
\label{triangle}
\end{figure}

Summing, we have
\begin{equation*}
E_n = 2C_{n-1} + E_{n-2} + A_{n-1}
\end{equation*}
for $n\ge 2$, with $E_0=1$ and $E_1=4$.  The solution to this recurrence relation is

\begin{equation}
\label{E}
\begin{split}
E_n &= \frac{17 + 7\sqrt{17}}{68} \left(\frac{3+\sqrt{17}}{2}\right)^n
    + \frac{17 - 7\sqrt{17}}{68} \left(\frac{3-\sqrt{17}}{2}\right)^n\\
    &\quad + \frac{2^{n+2} - (-1)^n}{6}.
\end{split}
\end{equation}  
This result was first derived, using different methods and different recurrences, 
by Goldwasser {\it et al.} \cite[Equation (21)]{Gold}.  This is sequence \href{https://oeis.org/A256960}{A256960} in the OEIS;
its first few terms are 1, 4,  11, 36, 119, 408, 1419, 4988, 17631, ... . 

From the discussion above we know that $F_n$,  the total number of cells in the first $2^n$ rows of the Pascal
rhombus (mod 2), is 
\begin{equation}
F_n =4^n,
\end{equation}
which is sequence \href{http://oeis.org/A000302}{A000302}. The density of ones in the first $2^n$ rows is then $E_n/F_n$, which approaches a limit of 0.

Now let $m$ be any positive integer, and let $n = \lfloor \log_2(m)\rfloor$, so that
$2^n \le m < 2^{n+1}$.  The first $m$ rows of the Pascal rhombus (mod 2) contain fewer than $E_{n+1}$ ones out of at least
$F_n$ total cells.  The density of ones in the first $m$ rows is thus less than $E_{n+1}/F_n$, which approaches 0 as a limit.
This is equivalent to the statement of Conjecture \ref{con4}.
\end{proof}

\section{Other work}
Moshe \cite{Moshe1} placed the Pascal rhombus (mod 2) in the general context of double linear recurrence 
sequences over finite fields.  Using extensions of the methods of Goldwasser {\it et al.} \cite{Gold}, he
proved that the number of ones, say, in the first $q^n$ rows of such a structure, where $q$ is the
order of the field, can be determined from the $n$-th power of a certain square matrix.  The 5$\times$5 matrix corresponding 
to the Pascal rhombus (mod 2) is explicitly displayed in Finch \cite{Finch}.  The expression in Equation (\ref{E}) for the
sequence $E_n$ can be derived from these two works.

In a later paper, Moshe \cite{Moshe2} proved that the number of ones in row $m$ of the structure can be determined from a product of matrices, one for each digit in the base-$q$ representation of $m$.  Again,  the relevant 
5$\times$5 matrices are displayed in Finch \cite{Finch}.  From these it is possible to compute the sequence 1, 3, 2, 5, 5, 6, 3, 11, 4, 15, 7, 10, \dots\ , (sequence \href{https://oeis.org/A059319}{A059319}), but no nice closed form expression for these numbers is known.

More recently, stealth configurations have been independently discovered by Sloane \cite{Neil} in his exploration of 2-dimensional cellular automata. In examining the evolution of the ``odd-rule'' cellular automaton using the centered von Neumann neighborhood and starting with a single ON cell at time 0, he found that the ON cells in generation $2^n-1$ form a diamond shaped
pattern $H_n$ consisting of four order $n-1$ stealth configurations $S_{n-1}$, one in each corner facing out, 
together with a copy of $H_{n-2}$ in the center.  Alternatively, we can view Sloane's $H_n$ pattern as the diamond formed
from the first $2^n$ rows
of the Pascal rhombus (mod 2) and its reflection in its bottom row.  Sloane calls our stealth configurations {\it haystacks}, and indicates the decomposition of a haystack into five smaller haystack in \cite[Line (35)]{Neil}.  The appearance of our
stealth configurations and his haystacks in rather different settings, analyzing different problems, seems quite remarkable.

%

\noindent\hrulefill

\footnotesize\noindent 2010  {\it Mathematics Subject Classification}: Primary 05A10, Secondary 11B50.

\footnotesize\noindent{\it Keywords}: Pascal, triangle, rhombus.

\end{document}